\documentclass[11pt]{article}
\usepackage{bbm}
\usepackage{mathrsfs}
\usepackage{amsfonts}
\usepackage{amssymb}
\usepackage{amsmath,theorem,float}
\usepackage{color}
\usepackage{multirow}
\usepackage{arydshln}

\newtheorem{theorem}{Theorem}[section]
\newtheorem{prop}[theorem]{Proposition}

\newtheorem{lemma}[theorem]{Lemma}
\newtheorem{definition}[theorem]{Definition}
\newtheorem{remark}[theorem]{Remark}
\newtheorem{example}[theorem]{Example}

\def\Z{{\mathbb{Z}}}

\def\N{{\mathbb{N}}}

\def\Y{{\mathbb{Y}}}

\def\bv{{\mathbf{v}}}

\def\rank{\hbox{\rm{rank}}}
\def\Span{\hbox{\rm{Span}}}
\def\D{{\sigma}}

\def\ord{\hbox{\rm{ord}}}
\def\fp{{\mathfrak{p}}}
\def\trdeg{\hbox{\rm{trdeg}}}
\def\Frac{\hbox{\rm{Frac}}}

\begin{document}
\title{Difference Index of Quasi-Prime Difference Algebraic Systems}
\author{Jie Wang}
\date{\today}
\maketitle

\begin{abstract}
This paper is devoted to studying difference indices of quasi-prime difference algebraic systems. We define the quasi dimension polynomial of a quasi-prime difference algebraic system. Based on this, we give the definition of the difference index of a quasi-prime difference algebraic system through a family of pseudo-Jacobian matrices. Some properties of difference indices are proved. In particular, an upper bound of difference indices is given. As applications, an upper bound of the Hilbert-Levin regularity and an upper bound of orders for difference ideal membership problem are deduced.
\end{abstract}

\section{Introduction}
There are several definitions of differential indices of a differential algebraic system in the literature (see for instance \cite{iqr,ldi,d-ca,d-le,d-pa,d-se}). Although they are not completely equivalent, in each case they represent a measure of the implicitness of the given system. In \cite{jie}, the difference index of a quasi-regular difference algebraic system was first defined. In this paper, we will generalize the definition of difference indices to more general difference algebraic systems, i.e. quasi-prime difference algebraic systems.

Suppose $F$ is a couple of difference polynomials, $\Delta$ is the difference ideal generated by $F$, and $\fp$ is a minimal reflexive prime difference ideal over $\Delta$. Denote $\Delta_k$ the algebraic ideal generated by $F$ and the transforms of $F$ with orders lower than $k$ in the corresponding localized polynomial ring at $\fp$. Then we say the system $F$ is {\em quasi-prime} at $\fp$ if $\Delta_k$ is a prime ideal for all $k\in\N^*$ and $\Delta$ is reflexive. For a difference algebraic system $F$ quasi-prime at $\fp$, we associate $F$ with a {\em $\fp$-quasi dimension polynomial}, which is a polynomial of degree one. By virtue of the $\fp$-quasi dimension polynomial, we can give the definition of the difference index of a quasi-prime difference algebraic system, which is called the {\em $\fp$-difference index}. As usual, its definition follows from a certain chain which eventually becomes stationary. Similarly to the case of $\mathfrak{P}$-differential indices in \cite{iqr} and the case of $\fp$-difference indices in \cite{jie}, the chain is established by the sequence of ranks of certain Jacobian submatrices associated with the system $F$. Assume $\omega$ is the $\fp$-difference index of the system $F$. It turns out that for $i\ge e-1$ ($e$ is the highest order of $F$), $\omega$ satisfies:
\[\Delta_{i-e+1+\omega}\cap A_i=\Delta\cap A_i,\]
where $A_i$ is the polynomial ring in the variables with orders no more than $i$, which meets our expectation for difference indices.

This approach enables us to give an upper bound of the $\fp$-difference index of a quasi-prime system. Based on this, we can give several applications of $\fp$-difference indices, including an upper bound of the Hilbert-Levin regularity and an upper bound of orders for difference ideal membership problem.

The paper will be organized as follows. In Section 2, we list some basic notions from difference algebra which will be used later. In Section 3, the $\fp$-quasi dimension polynomial of a quasi-prime difference algebraic system is defined. In Section 4, we introduce a family of pseudo-Jacobian matrices and give the definition of $\fp$-difference indices through studying the ranks of them. In Section 5, some properties of $\fp$-difference indices will be proved. In Section 6, several applications of $\fp$-difference indices are given. In Section 7, we give an example.

\section{Preliminaries}
A {\em difference ring} or {\em $\sigma$-ring} for short $(R,\sigma)$, is a commutative ring $R$ together with a ring endomorphism $\sigma\colon R\rightarrow R$. If $R$ is a field, then we call it a {\em difference field}, or a {\em $\sigma$-field} for short. We usually omit $\sigma$ from the notation, simply refer to $R$ as a $\sigma$-ring or a $\sigma$-field. In this paper, $K$ is always assumed to be a $\D$-field of characteristic $0$.

\begin{definition}
Let $R$ be a $\D$-ring. An ideal $I$ of $R$ is called a {\em $\D$-ideal} if for $a\in R$, $a\in I$ implies $\D(a)\in I$. Suppose $I$ is a $\D$-ideal of $R$, then $I$ is called
\begin{itemize}
  \item {\em reflexive} if $\D(a)\in I$ implies $a\in I$ for $a\in R$;
  \item {\em $\D$-prime} if $I$ is reflexive and a prime ideal as an algebraic ideal.
\end{itemize}
\end{definition}

For a subset $F$ in a $\D$-ring, we denote $[F]$ the $\D$-ideal generated by $F$. Let $K$ be a $\sigma$-field. Suppose $\Y=\{y_1,\ldots,y_n\}$ is a set of $\sigma$-indeterminates over $K$. Then the {\em $\sigma$-polynomial ring} over $K$ in $\Y$ is the polynomial ring in the variables $\Y,\sigma(\Y),\sigma^2(\Y),\ldots$. It is denoted by
\[K\{\Y\}=K\{y_1,\ldots,y_n\}\]
and has a natural $K$-$\sigma$-algebra structure. For more details about difference algebra, one can refer to \cite{wibmer}.

For the later use, we give the classical Jacobian Criterion here.
\begin{lemma}[Jacobian Criterion]\label{pd-lm}
Let $S=K[y_1,\ldots,y_n]$ be the polynomial ring over $K$. Let $I=(f_1,\ldots,f_r)$ be an ideal of $S$ and set $R=S/I$. Let $P$ be a prime ideal of $S$ containing $I$ and assume $\kappa(P)$ is the residue class field of $P$. Then
\[\dim_{\kappa(P)}\kappa(P)\otimes\Omega_{R_P/K}=n-\rank_{\kappa(P)}J,\]
where $J:=(\partial f_i/\partial y_j)_{r\times n}$ is the Jacobian matrix. In particular, if $I$ is itself a prime ideal, then
$\dim_{\kappa(I)}\Omega_{\kappa(I)/K}=n-\rank_{\kappa(I)}J,$
where $\kappa(I)$ is the residue class field of $I$.
\end{lemma}
\begin{proof}
One can find a proof in \cite[Chapter 16, Theorem 16.19]{es-ca}.
\end{proof}

\section{Quasi-prime difference algebraic systems}
Let $K$ be a $\D$-field. Let $a$ be an element in a $\D$-extension field of $K$, $S$ a set of elements in a $\D$-extension field of $K$, and $i\in\N$. Denote $a^{(i)}=\D^i(a), a^{[i]}=\{a,a^{(1)},\ldots,a^{(i)}\}$, $S^{(i)} = \cup_{a\in S}\{a^{(i)}\}$ and $S^{[i]} = \cup_{a\in S} a^{[i]}$. For the $\sigma$-indeterminates $\Y=\{y_1,\ldots,y_n\}$ and $i\in\N$, we will treat the elements of $\Y^{[i]}$ as algebraic indeterminates, and $K[\Y^{[i]}]$ is the polynomial ring in $\Y^{[i]}$.
%For a $\D$-ideal $\I\subset k\{\Y\}$ we define $$\I_i:= \I\cap k[\Y^{[i]}].$$

Throughout the paper let $F=\{f_1,\ldots,f_r\}\subset K\{\Y\}$ be a system of difference polynomials over $K$, $[F]$ the $\D$-ideal generated by $F$, and $\fp\subseteq K\{\Y\}$ a $\D$-prime ideal minimal over $[F]$. Let $\epsilon_{ij}:=\ord_{y_j}(f_i)$ which is the order of $f_i$ with respect to $y_j$ and denote $e:=\max\{\epsilon_{ij}\}$ for the maximal order of transforms which occurs in $F$. We assume that $F$ actually involves difference operator, i.e. $e\ge1$.
%Let $e_i:=\max_{j=1}^n\{\epsilon_{ij}\}\in\N$, i.e. the order of the highest difference of a variable in $\Y$ appearing in this difference polynomial.
We introduce also the following auxiliary polynomial rings and ideals: for every $k\in\N$, $A_k$ denotes the polynomial ring $A_k:=K[\Y^{[k]}]$ and $\Delta_k:=(f_1^{[k-1]},\ldots,f_r^{[k-1]})\subseteq A_{k-1+e}$. We set $\Delta_0:=(0)$ by definition.

For each non-negative integer $k$ we write $B_k$ for the local ring obtained from $A_k$ after localization at the prime ideal $A_k\cap\fp$ and we denote $\fp_k:=A_{k-1+e}\cap\fp$. For the sake of simplicity, we preserve the notation $\Delta_k$ for the ideal generated by $f_1^{[k-1]},\ldots,f_r^{[k-1]}$ in the local ring $B_{k-1+e}$ and denote $\Delta$ the $\D$-ideal generated by $F$ in $K\{\Y\}_{\fp}$.
\begin{definition}
We say that the system $F$ is {\em quasi-prime} at $\fp$ if $\Delta_k$ is a prime ideal in the ring $B_{k-1+e}$ for all $k\in\N$ and $\Delta$ is reflexive.
\end{definition}

If the system $F$ is quasi-prime at $\fp$, then by the minimality of $\fp$, $\Delta$ agrees with $\fp$ in $K\{\Y\}_{\fp}$, since $\Delta$ itself is a $\D$-prime ideal.
\begin{remark}\label{pd-re}
If the $\D$-ideal $[F]\subseteq K\{\Y\}$ is already a $\D$-prime ideal, the minimality of $\fp$ implies that $\fp=[F]$ and all our results remain true considering the rings $A_k$ and the $\D$-ideal $[F]$ without localization. In this case if $F$ is quasi-prime at $[F]$ we will say simply that $F$ is quasi-prime.
\end{remark}

In this paper, we always assume that $F$ is a difference algebraic system which is quasi-prime at $\fp$.

For a matrix $E$ over $K$, we use $E^{(i)}$ to denote the matrix whose elements are the $i$-th transforms of the corresponding elements of $E$.

\begin{lemma}\label{pd-lm2}
For a matrix $E$ over $K$, $\rank(E^{(1)})=\rank(E)$.
\end{lemma}
\begin{proof}
It is clear that the maximal nonzero minors of $E^{(1)}$ and $E$ have the same order since the difference operator on $K$ is injective. It follows that $\rank(E^{(1)})=\rank(E)$.
\end{proof}

\begin{lemma}\label{pd-lemma1}
Let $E_1,E_2,\ldots,E_t\in K^{p\times q}$ and
\[M_k:=\begin{pmatrix}
E_1&E_2&\cdots&E_t&&&\\
&E_1^{(1)}&E_2^{(1)}&\cdots&E_t^{(1)}&&\\
&&\ddots&\ddots&\ddots&\ddots&\\
&&&E_1^{(k-1)}&E_2^{(k-1)}&\cdots&E_t^{(k-1)}
\end{pmatrix}.\]
Then for $k$ large enough, there exists $d'\in\N$ and $s'\in\Z$ such that
\begin{equation}\label{pd-eq1}
\rank(M_k)=d'k+s'.
\end{equation}
Moreover, the least $k$ such that the equality (\ref{pd-eq1}) holds is bounded by $(t-1)(\min\{p,q\}+1)$.
\end{lemma}
\begin{proof}
For the sake of convenience, for each pair $m,n\in\N, m\le n$, let us define an operator $\pi_n^m$ on subspaces of $K^n$, $$\pi_n^m(V):=\{\bv\in K^m\mid (\mathbf{0},\bv)\in V, \mathbf{0}\in K^{n-m}\},$$
where $V$ is a subspace of $K^n$.

Suppose $k\ge t-1$. We will apply the Gaussian elimination method to $M_k$ with some changes. First, apply the Gaussian elimination method to the submatrix of $M_k$
\[C:=\left( \begin{array}{cccc:cccc:}
\cdashline{5-8}
E_1&E_2&\cdots&\cdots&E_t&&&\\
&E_1^{(1)}&E_2^{(1)}&\cdots&E_{t-1}^{(1)}&E_t^{(1)}&&\\
&&\ddots&\ddots&\ddots&\ddots&\ddots&\\
&&&E_1^{(t-2)}&E_2^{(t-2)}&\cdots&\cdots&E_t^{(t-2)}\\
\cdashline{5-8}
\end{array} \right),\]
and denote the resulting matrix by $A$. Then we obtain several rows of $A$ containing nonzero elements only in the dotted line area, with the first 1's of the corresponding rows lying in distinct columns. Let $i_1,\ldots,i_m$ be the row indices of these rows of $A$. Let $B$ be the submatrix of $A$ obtained by removing the first $(t-1)q$ columns and these rows whose row indices are not $i_1,\ldots,i_m$ from $A$. Let $U_0$ be subspace of $K^{(t-1)q}$ spanned by the row vectors of $B$.

Now first perform row reductions to the next block matrix $\begin{pmatrix}
E_1^{(t-1)}&E_2^{(t-1)}&\cdots&E_t^{(t-1)}\end{pmatrix}$ by using the row vectors of $B$, and then apply the Gaussian elimination method to the resulting matrix itself. We again obtain some rows containing nonzero elements only in the dotted line area:
\[\left( \begin{array}{cccc:cccc:}
\cdashline{5-8}
E_1^{(1)}&E_2^{(1)}&\cdots&\cdots&E_t^{(1)}&&&\\
&E_1^{(2)}&E_2^{(2)}&\cdots&E_{t-1}^{(2)}&E_t^{(2)}&&\\
&&\ddots&\ddots&\ddots&\ddots&\ddots&\\
&&&E_1^{(t-1)}&E_2^{(t-1)}&\cdots&\cdots&E_t^{(t-1)}\\
\cdashline{5-8}
\end{array} \right).\]
As before, let the fragments of these rows in the dotted line area span the subspace $U_1\subseteq K^{(t-1)q}$. Denote the subspace spanned by the rows of the submatrix $C$ by $W$. We see that $U_0=\pi_{(2t-2)q}^{(t-1)q}(W)$ and $U_1=\pi_{(2t-1)q}^{(t-1)q}(P)$, where $P$ is the subspace spanned by the row vectors of $E_1E_2\cdots E_t\times 0^{p\times (t-1)q}$ and the vectors in $\{0\}^q\times W^{(1)}.$ It follows that $U_0^{(1)}\subseteq U_1$.

Denote the row vectors of the submatrix 
$$\begin{pmatrix}E_1^{(j+t-2)}&E_2^{(j+t-2)}&\cdots&E_t^{(j+t-2)}\end{pmatrix}$$
by $V_j$, and then $V_{j+1}=V_j^{(1)}$. Perform the procedure as above, and we recursively define $$U_{j}:=\pi_{tq}^{(t-1)q}(\Span(U_{j-1}\times \{0\}^{q}\cup V_{j}))$$
for $j\ge 1$. We will show that $U_j^{(1)}\subseteq U_{j+1}$ for all $j\ge 0$ and if $U_j^{(1)}=U_{j+1}$, then $U_{j+1}^{(1)}=U_{j+2}$.

Let us do induction on $j$. The case $j=0$ has proved above. Now suppose $j\ge 1$. Then by the induction hypothesis, $U_j^{(1)}=\pi_{tq}^{(t-1)q}(\Span(U_{j-1}^{(1)}\times \{0\}^{q}\cup V_{j}^{(1)})) \subseteq\pi_{tq}^{(t-1)q}(\Span(U_{j}\times \{0\}^{q}\cup V_{j+1}))=U_{j+1}$, and if $U_{j-1}^{(1)}=U_{j}$, then $U_{j}^{(1)}=U_{j+1}$. So by Lemma \ref{pd-lm2}, $\dim(U_j)=\dim(U_j^{(1)})\le\dim(U_{j+1})$, and if $\dim(U_j)=\dim(U_{j+1})$, then $\dim(U_{j+1})=\dim(U_{j+2})$. It follows that $(\dim(U_j))_{j\in\N}$ is a non-decreasing sequence and eventually stabilizes at some constant at most $(t-1)\min\{p,q\}+1$ steps since the dimensions of the subspaces $U_j$ are no larger than $(t-1)\min\{p,q\}$. So there exists a non-negative integer $r\le(t-1)\min\{p,q\}$ such that for $j\ge r$, $\dim(U_j)=\dim(U_{j+1})$ and $\dim(\Span(U_{j}\times \{0\}^{q}\cup V_{j+1}))=\dim(\Span(U_{j+1}\times \{0\}^{q}\cup V_{j+2}))$ by Lemma \ref{pd-lm2}. As a consequence, the rank of the corresponding matrix $M_{j+t-1}$ will increase by a constant at each step. That is to say, for $k$ large enough, there exist $d',s'\in\N$ such that
\[\rank(M_k)=d'k+s',\]
and the least $k$ such that the above equality holds is bounded by $(t-1)\min\{p,q\}+t-1=(t-1)(\min\{p,q\}+1)$.
\end{proof}

Let us define
\begin{align*}
J_k:&=\frac{\partial{(F,F^{(1)},\ldots,F^{(k-1)}})}{\partial{(\Y,\Y^{(1)},\ldots,\Y^{(k-1+e)})}}\\
&=\begin{pmatrix}
\frac{\partial F}{\partial\Y}&\frac{\partial F}{\partial\Y^{(1)}}&\cdots&\frac{\partial F}{\partial\Y^{(e)}}&&&\\
&\frac{\partial F^{(1)}}{\partial\Y^{(1)}}&\frac{\partial F^{(1)}}{\partial\Y^{(2)}}&\cdots&\frac{\partial F^{(1)}}{\partial\Y^{(e+1)}}&&\\ &&\ddots&\ddots&\ddots&\ddots&\\
&&&\frac{\partial F^{(k-1)}}{\partial\Y^{(k-1)}}&\frac{\partial F^{(k-1)}}{\partial\Y^{(k)}}&\cdots&\frac{\partial F^{(k-1)}}{\partial\Y^{(k-1+e)}}
\end{pmatrix},
\end{align*}
where each $\frac{\partial F^{(p)}}{\partial \Y^{(q)}}$ denotes the Jacobian matrix $(\partial(f_1^{(p)},\ldots,f_r^{(p)})/\partial(y_1^{(q)},\ldots,y_n^{(q)}))_{r\times n}$.

Since the partial derivative operator and the difference operator are commutative, we have
\begin{equation*}
J_k=\begin{pmatrix}
\frac{\partial F}{\partial\Y}&\frac{\partial F}{\partial\Y^{(1)}}&\cdots&\frac{\partial F}{\partial\Y^{(e)}}&&&\\
&(\frac{\partial F}{\partial\Y})^{(1)}&(\frac{\partial F}{\partial\Y^{(1)}})^{(1)}&\cdots&(\frac{\partial F}{\partial\Y^{(e)}})^{(1)}&&\\ &&\ddots&\ddots&\ddots&\ddots&\\
&&&(\frac{\partial F}{\partial\Y})^{(k-1)}&(\frac{\partial F}{\partial\Y^{(1)}})^{(k-1)}&\cdots&(\frac{\partial F}{\partial\Y^{(e)}})^{(k-1)}
\end{pmatrix}.
\end{equation*}

Denote $\kappa(\Delta_k)$ the residue class field of $\Delta_k$ in the ring $B_{k-1+e}$, $\kappa(\fp_k)$ the residue class field of $\fp_k$ in the ring $A_{k-1+e}$ and $\kappa$ the residue class field of $\fp$. To define the $\fp$-quasi dimension polynomial of the system $F$, we assume that the rank of the matrix $J_k$ over $\kappa(\Delta_{k+i})$ does not depend on $i$, where $i\in\N$. That is to say, the rank of the matrix $J_k$ considered alternatively over $\kappa(\Delta_k)$, or over $\kappa(\fp_k)$, or over $\kappa$ is always the same.
\begin{theorem}\label{pd-thm}
Suppose $F$ is a difference algebraic system which is quasi-prime at $\fp$. Let $\psi(k):=\trdeg_K(\kappa(\Delta_k))$. Then for $k$ large enough, there exists $d\in\N$ and $s\in\Z$ such that
\[\psi(k)=dk+s.\]
Moreover, the least $k$ such that the above equality holds is bounded by $e(\min\{r,n\}+1)$.
\end{theorem}
\begin{proof}
By the property of K\"ahler differentials, $\psi(k)=\trdeg_K(\kappa(\Delta_k))=\dim_{\kappa(\Delta_k)}\Omega_{\kappa(\Delta_k)/K}$. By Lemma \ref{pd-lm}, $\dim_{\kappa(\fp_k)}\kappa(\fp_k)\otimes\Omega_{\kappa(\Delta_k)/K}=\dim_{\kappa(\Delta_k)}\Omega_{\kappa(\Delta_k)/K}
=(k+e)n-\rank_{\kappa(\fp_k)}(J_k)=(k+e)n-\rank_{\kappa}(J_k)$. It follows $\psi(k)=(k+e)n-\rank_{\kappa}(J_k)$. Thus the conclusions of the theorem follow from Lemma \ref{pd-lemma1} by setting $d=n-d'$ and $s=s'+en$.
\end{proof}
\begin{definition}
In the above theorem, $\psi(k)=dk+s$ is called the {\em $\fp$-quasi dimension polynomial} of the system $F$, and the least $k$ such that the $\fp$-quasi dimension polynomial holds is called the {\em $\fp$-quasi regularity degree} of $F$, which is denoted by $\rho$.
\end{definition}

\section{The definition of $\fp$-difference index}
Following \cite{iqr}, we introduce a family of pseudo-Jacobian matrices which we need in order to define the concept of difference index.
\begin{definition}
For each $k\in\N$ and $i\in\N_{\ge e-1}$ (i.e. $i\in\N$ and $i\ge e-1$), we define the $kr\times kn$-matrix $J_{k,i}$ as follows:
\begin{align*}
J_{k,i}:&=\frac{\partial{(F^{(i-e+1)},F^{(i-e+2)},\ldots,F^{(i-e+k)}})}{\partial{(\Y^{(i+1)},\Y^{(i+2)},\ldots,\Y^{(i+k)})}}\\
&=\begin{pmatrix}
\frac{\partial F^{(i-e+1)}}{\partial \Y^{(i+1)}}&0&\cdots&0\\ \frac{\partial F^{(i-e+2)}}{\partial \Y^{(i+1)}}&\frac{\partial F^{(i-e+2)}}{\partial \Y^{(i+2)}}&\cdots&0\\ \vdots&\vdots&\ddots&\vdots\\ \frac{\partial F^{(i-e+k)}}{\partial \Y^{(i+1)}}&\frac{\partial F^{(i-e+k)}}{\partial \Y^{(i+2)}}&\cdots&\frac{\partial F^{(i-e+k)}}{\partial \Y^{(i+k)}}
\end{pmatrix},
\end{align*}
where each $\frac{\partial F^{(p)}}{\partial \Y^{(q)}}$ denotes the Jacobian matrix $(\partial(f_1^{(p)},\ldots,f_r^{(p)})/\partial(y_1^{(q)},\ldots,y_n^{(q)}))_{r\times n}$.
\end{definition}

Since the partial derivative operator and the difference operator are commutative, we have
\[J_{k,i}=\begin{pmatrix}
(\frac{\partial F}{\partial \Y^{(e)}})^{(i-e+1)}&0&\cdots&0\\ (\frac{\partial F}{\partial \Y^{(e-1)}})^{(i-e+2)}&(\frac{\partial F}{\partial \Y^{(e)}})^{(i-e+2)}&\cdots&0\\ \vdots&\vdots&\ddots&\vdots\\ (\frac{\partial F}{\partial \Y^{(e-k+1)}})^{(i-e+k)}&(\frac{\partial F}{\partial \Y^{(e-k+2)}})^{(i-e+k)}&\cdots&(\frac{\partial F}{\partial \Y^{(e)}})^{(i-e+k)}
\end{pmatrix},\]
where we set that $\frac{\partial F}{\partial \Y^{(j)}}=0$ if $j<0$.

Note that $J_{k,i+1}=J_{k,i}^{(1)}$.

\begin{definition}
For $k\in\N$ and $i\in\N_{\ge e-1}$, we define $\mu_{k,i}\in\N$ as follows:
\begin{itemize}
  \item $\mu_{0,i}:=0$;
  \item $\mu_{k,i}:=\dim_{\kappa}\ker(J_{k,i}^{\tau})$, for $k\ge1$, where $J_{k,i}^{\tau}$ denotes the usual transpose of the matrix $J_{k,i}$. In particular $\mu_{k,i}=kr-\rank_{\kappa}(J_{k,i})$.
\end{itemize}
\end{definition}

\begin{prop}
Let $k\in\N$ and $i\in\N_{\ge e-1}$. Then $\mu_{k,i}=\mu_{k,i+1}$.
\end{prop}
\begin{proof}
Since $J_{k,i+1}=J_{k,i}^{(1)}$ for any $k\in\N$ and any $i\in\N_{\ge e-1}$, then $\mu_{k,i}=\mu_{k,i+1}$ follows from Lemma \ref{pd-lm2}.
\end{proof}

The previous proposition shows that the sequence $\mu_{k,i}$ does not depend on the index $i$. Therefore, in the sequel, we will write $\mu_k$ instead of $\mu_{k,i}$, for any $i\in\N_{\ge e-1}$.

For $k\in\N$ and $i\in\N_{\ge e-1}$, we denote $\Omega_{i,k}$ the residue class field of $\Delta_{i-e+1+k}\cap B_i$ in the ring $B_i$. As an additional hypothesis on the system $F$, we assume that the rank of the matrix $J_{k,i}$ over $\kappa(\Delta_{i-e+1+k+s})$ does not depend on $s$, where $s\in\N$. That is to say, we assume that the rank of the matrix $J_{k,i}$ considered alternatively over $\kappa(\Delta_{i-e+1+k})$, or over $\kappa(\fp_{i-e+1+k})$, or over $\kappa$ is always the same.
\begin{prop}\label{di-prop1}
Assume the $\fp$-quasi dimension polynomial of $F$ is $\psi(k)=dk+s$ and the $\fp$-quasi regularity degree is $\rho$. Let $k\in\N$ and $i\in\N_{\ge e-1}$. Then
\begin{enumerate}
  \item The transcendence degree of the field extension
  $$\Frac(B_{i}/(\Delta_{i-e+1+k}\cap B_{i}))\hookrightarrow \Frac(B_{i+k}/\Delta_{i-e+1+k})$$
  is $k(n-r)+\mu_{k}$.
  \item For $i+k\ge\rho+e-1$, the following identity holds:
  $$\trdeg_K(\Frac(B_{i}/(\Delta_{i-e+1+k})\cap B_{i}))=d(i+1)+(d+r-n)k+s-ed-\mu_{k}.$$
\end{enumerate}
\end{prop}
\begin{proof}
\begin{enumerate}
  \item We can consider the fraction field of $\Frac(B_{i+k}/\Delta_{i-e+1+k})$ as the fraction field of $$\Omega_{i,k}[\Y^{(i+1)},\ldots,\Y^{(i+k)}]/(F^{(i-e+1)},\ldots,F^{(i-e+k)}).$$
       Therefore by Lemma \ref{pd-lm}, the transcendence degree of the field extension equals $kn-\rank_{\kappa}(J_{k,i})=kn-(kr-\mu_k)=k(n-r)+\mu_k.$
  \item Since when $i+k\ge\rho+e-1$, by Theorem \ref{pd-thm}, $\trdeg_K(\Frac(B_{i+k}/\Delta_{i-e+1+k}))=d(i-e+1+k)+s$, we have
  \begin{align*}
  \trdeg_K(\Frac(B_{i}/(\Delta_{i-e+1+k}\cap B_{i})))&=d(i-e+1+k)+s-k(n-r)-\mu_k\\
  &=d(i+1)+(d+r-n)k+s-ed-\mu_{k}.
  \end{align*}
\end{enumerate}
\end{proof}

\begin{lemma}\label{di-lemma2}
Let $E_1,E_2,\ldots,E_t\in K^{p\times q}$ and
\[N_k:=\begin{pmatrix}
E_1&&&&&\\
E_2^{(1)}&E_1^{(1)}&&&&\\
\vdots&\vdots&\ddots&&&\\
E_t^{(t-1)}&E_{t-1}^{(t-1)}&\cdots&E_1^{(t-1)}&&\\
&\ddots&\ddots&\ddots&\ddots&\\
&&E_t^{(k-1)}&E_{t-1}^{(k-1)}&\cdots&E_1^{(k-1)}
\end{pmatrix}.\]
Then for $k$ large enough, there exists $d'\in\N$ and $a'\in\Z$ such that
\[\rank(N_k)=d'k+a'.\]
Moreover, the least $k$ such that the above equality holds is bounded by $(t-1)(\min\{p,q\}+2)$.
\end{lemma}
\begin{proof}
Assume $k\ge 2t-2$. %The transpose of $N_k$ is 
%\[N_k^{\tau}:=\begin{pmatrix}
%E_1^{\tau}&(E_2^{\tau})^{(1)}&\cdots&(E_t^{\tau})^{(t-1)}&&\\
%&(E_1^{\tau})^{(1)}&(E_2^{\tau})^{(2)}&\cdots&(E_t^{\tau})^{(t)}&\\
%&\ddots&\ddots&\ddots&\ddots&\\
%&&(E_1^{\tau})^{(k-t)}&(E_2^{\tau})^{(k-t+1)}&\cdots&(E_t^{\tau})^{(k-1)}\\
%&&&(E_1^{\tau})^{(k-t+1)}&\cdots&(E_{t-1}^{\tau})^{(k-1)}\\
%&&&&\ddots&\vdots\\
%&&&&&(E_1^{\tau})^{(k-1)}
%\end{pmatrix}.\]
Let us consider the submatrix of $N_k$
\[C_k:=\begin{pmatrix}
E_t^{(t-1)}&E_{t-1}^{(t-1)}&\cdots&E_1^{(t-1)}&&&\\
&E_t^{(t)}&E_{t-1}^{(t)}&\cdots&E_1^{(t)}&&\\
&&\ddots&\ddots&\ddots&\ddots&\\
&&&E_t^{(k-1)}&E_{t-1}^{(k-1)}&\cdots&E_1^{(k-1)}
\end{pmatrix}.\]
Similarly to the proof of Lemma \ref{pd-lemma1}, apply the Gaussian elimination method to $C_k$, but from bottom to top, from right to left. Then we have that for $k$ large enough, there exists $d'\in\N$ and $s'\in\Z$ such that
\[\rank(C_k)=d'(k-t+1)+s',\]
and the least $k$ such that $\rank(C_k)=d'(k-t+1)+s'$ is bounded by $(t-1)(\min\{p,q\}+1)+t-1=(t-1)(\min\{p,q\}+2)$.
And as a consequence, for $k$ large enough, there exists a constant $c\in\N$ such that $\rank(N_k)=\rank(C_k)+c=d'(k-t+1)+s'+c$. Set $a'=-d'(t-1)+s'+c$, then for $k$ large enough, $\rank(N_k)=d'k+a'$ and the least $k$ such that $\rank(N_k)=d'k+a'$ is bounded by $(t-1)(\min\{p,q\}+2)$.
\end{proof}

Due to the above lemma, we can prove a formula of $\mu_k$ for $k\gg0$.
\begin{theorem}\label{di-thm1}
Suppose $F$ is a difference algebraic system which is quasi-prime at $\fp$. Assume the $\fp$-quasi dimension polynomial of $F$ is $\psi(k)=dk+s$. Then for $k\gg0$, there exists $a\in\Z$ such that
\[\mu_k=(d+r-n)k+a.\]
Moreover, an upper bound of the least $k$ such that the above equality holds is $e(\min\{r,n\}+2)$.
\end{theorem}
\begin{proof}
Set $i=e-1$. Then for $k\gg0$, 
\[J_{k,e-1}=\begin{pmatrix}
\frac{\partial F}{\partial \Y^{(e)}}&&&&&\\
(\frac{\partial F}{\partial \Y^{(e-1)}})^{(1)}&(\frac{\partial F}{\partial \Y^{(e)}})^{(1)}&&&&\\
\vdots&\vdots&\ddots&&&\\
(\frac{\partial F}{\partial \Y})^{(e)}&(\frac{\partial F}{\partial \Y^{(1)}})^{(e)}&\cdots&(\frac{\partial F}{\partial \Y^{(e)}})^{(e)}&&\\
&\ddots&\ddots&\ddots&\ddots&\\
&&(\frac{\partial F}{\partial \Y})^{(k-1)}&(\frac{\partial F}{\partial \Y^{(1)}})^{(k-1)}&\cdots&(\frac{\partial F}{\partial \Y^{(e)}})^{(k-1)}
\end{pmatrix}.\]
So by Lemma \ref{di-lemma2}, for $k\gg0$, there exists $d'\in\N$ and $a'\in\Z$ such that $\rank(J_{k,e-1})=d'k+a'$, and the least $k$ such that $\rank(J_{k,i})=d'k+a'$ is bounded by $e(\min\{r,n\}+2)$. Note that $d'=n-d$. Set $a=-a'$. Hence for $k\gg0$, $\mu_k=kr-\rank(J_{k,e-1})=(d+r-n)k+a$, and an upper bound of the least $k$ such that $\mu_k=(d+r-n)k+a$ is $e(\min\{r,n\}+2)$.
\end{proof}

\begin{remark}
Let $\rho$ be the $\fp$-quasi regularity degree of the system $F$. From the proof, we actually have a more accurate upper bound for the least $k$ such that $\mu_k=(d+r-n)k+a$, namely, $\rho+e$.
\end{remark}
\begin{remark}
In fact, we can deduce the formula of $\mu_k$ for $k\gg0$ in a more straightforward way.
Fix an index $i\in\N_{\ge e-1}$. By Proposition \ref{di-prop1}, we have $\psi(i-e+1+k)=k(n-r)+\mu_k+\trdeg_K(\Omega_{i,k})$.
Note that $\trdeg_K(\Omega_{i,k})$ is a constant for $k\gg0$ since the increasing chain $(\Delta_{i-e+1+k}\cap B_i)_{k\in\N}$ of prime ideals in the ring $B_i$ is stable. So by Theorem \ref{pd-thm}, $\mu_k$ is a polynomial of degree one for $k\gg0$.
\end{remark}

\begin{definition}
In the above theorem, the least integer $k$ such that $\mu_k=(d+r-n)k+a$ is called the {\em $\fp$-difference index} of the system $F$, which is denoted by $\omega$. If $[F]$ is itself a $\D$-prime ideal, we say simply the difference index of $F$.
\end{definition}

It is obvious from the construction that $\omega$ is depending on the choice of the minimal $\D$-prime ideal $\fp$ over $[F]$. However, we will prove some properties of $\omega$ which meet our expectation for difference indices.

\section{Properties of $\fp$-difference index}
A notable property of most differentiation indices is that they provide an upper bound for the number of derivatives of the system needed to obtain all the equations that must be satisfied by the solutions of the system. This case is also suitable for the $\fp$-difference indices defined above.
\begin{theorem}\label{mc-tm}
Suppose $F$ is a difference algebraic system which is quasi-prime at $\fp$. Let $\rho$ and $\omega$ be the $\fp$-quasi regularity degree and the $\fp$-difference index of the system $F$ respectively. Then, for $i\in\N_{\ge e-1}$ such that $i+\omega\ge\rho+e-1$, the equality of ideals
\[\Delta_{i-e+1+\omega}\cap B_i=\Delta\cap B_i\]
holds in the ring $B_i$. Moreover, for every $i\in\N_{\ge e-1}$, let $h_i:=\min\{h\in\N:\Delta_{i-e+1+h}\cap B_i=\Delta\cap B_i\}$. Then if $i+\omega\ge\rho+e-1$ and $i+h_i\ge\rho+e-1$, then $\omega=h_i$.
\end{theorem}
\begin{proof}
Fix the index $i\in\N_{\ge e-1}$ such that $i+\omega\ge\rho+e-1$. Let us consider the increasing chain $(\Delta_{i-e+1+k}\cap B_i)_{k\in\N}$ of prime ideals in the ring $B_i$. From Proposition \ref{di-prop1}, for $i+k\ge\rho+e-1$, we have
\begin{equation}\label{pdi-eq}
\trdeg_k(\Frac(B_{i}/(\Delta_{i-e+1+k}\cap B_{i})))=d(i+1)+(d+r-n)k+s-ed-\mu_{k}.
\end{equation}
Since $\mu_k=(d+r-n)k+a$ for $k\ge\omega$ by Theorem \ref{di-thm1}, $\trdeg_k(\Frac(B_{i}/(\Delta_{i-e+1+k}\cap B_{i})))=d(i+1)+s-ed-a$ for $k\ge\omega$. So all of the prime ideals $\Delta_{i-e+1+k}\cap B_{i}$ have the same dimension for $k\ge\omega$ and the chain of prime ideals becomes stationary for $k\ge\omega$.

It only remains to prove that the largest ideal of the chain coincides with $\Delta\cap B_i$. One inclusion is obvious. For the other, let $f$ be an arbitrary element of $\Delta\cap B_i$, then there exist difference polynomials $h,a_{lj}\in K\{\Y\}, h\notin\fp$ such that
\[f=\sum^r_{l=1}\sum_{j}\frac{a_{lj}f^{(j)}_l}{h}.\]
Let $N$ be the maximal order of the variables $\Y$ appearing in this equality. Then we have $f\in \Delta_{N-e+1}\subseteq B_N$ and hence $f\in\Delta_{N-e+1}\cap B_i$. Since the above chain of ideals is stationary for $k\ge\omega$, $f\in\Delta_{i-e+1+\omega}\cap B_i$. This completes the proof of the first assertion of the Theorem.

For the second part of the statement, by the definition of $h_i$, the transcendence degrees $\trdeg_K(\Frac(B_{i}/(\Delta_{i-e+1+k}\cap B_{i})))$ coincide for $k\ge h_i$, and hence by (\ref{pdi-eq}), $\mu_k$ becomes a polynomial of degree one for $k\ge h_i$. This implies that $\omega\le h_i$. The equality follows from the first part of the statement and the minimality of $h_i$.
\end{proof}

\begin{remark}
Taking $i=e-1$ in the last assertion of the above theorem, then if $\omega\ge\rho$ and $h_{e-1}\ge\rho$, we have the following equality for the $\fp$-difference index:
\[\omega=\min\{h\in\N:\Delta_{h}\cap B_{e-1}=\Delta\cap B_{e-1}\}\]
\end{remark}

The following proposition reveals a connection between the formula of $\mu_k$ for $k\gg0$ and the dimension polynomial of $\fp$.
\begin{prop}\label{di-prop2}
Assume the $\fp$-quasi dimension polynomial of the system $F$ is $\psi(k)=dk+s$ and for $k\gg0$, $\mu_k=(d+r-n)k+a.$ Then $d=\D\textrm{-}\dim(\fp)$ and $a=s-ed-\ord(\fp)$, where $\D\textrm{-}\dim(\fp)$ and $\ord(\fp)$ are the difference dimension and the order of $\fp$ respectively. In particular, if $\omega$ is the $\fp$-difference index of the system $F$, then $\mu_{\omega}=(d+r-n)\omega+s-ed-\ord(\fp).$
\end{prop}
\begin{proof}
Let $\rho$ be the $\fp$-quasi regularity degree of the system $F$. Fix an index $i\in\N_{\ge e-1}$ such that $i+\omega\ge\rho+e-1$. By Theorem \ref{mc-tm}, for $k\ge\omega$, $\Delta_{i-e+1+k}\cap B_i=\Delta\cap B_i$. Therefore, for $k\ge\omega$, by Proposition \ref{di-prop1} and Theorem \ref{di-thm1},
\begin{align*}
\trdeg_K(\Frac(B_{i}/(\Delta\cap B_{i})))&=\trdeg_K(\Frac(B_{i}/(\Delta_{i-e+1+k}\cap B_{i})))\\
&=d(i+1)+(d+r-n)k+s-ed-\mu_{k}\\
&=d(i+1)+s-ed-a.
\end{align*}
On the other hand, since $\Frac(B_{i}/(\Delta\cap B_{i}))=\Frac(A_{i}/(\fp\cap A_{i}))$, by the dimension polynomial of $\fp$ (see for instance \cite[Chapter 5]{wibmer}),
$$\trdeg_K(\Frac(B_{i}/(\Delta\cap B_{i})))=\D\textrm{-}\dim(\fp)(i+1)+\ord(\fp).$$ So
\begin{equation}\label{di-equ1}
d(i+1)+s-ed-a=\D\textrm{-}\dim(\fp)(i+1)+\ord(\fp)
\end{equation}
for all $i\in\N_{e-1}$ such that $i+\omega\ge\rho+e-1$. Compare the coefficients of $i$ on the two sides of the identity (\ref{di-equ1}), and it follows $d=\D\textrm{-}\dim(\fp)$ and $a=s-ed-\ord(\fp)$.
\end{proof}

\begin{remark}
Note that $\Delta_{i-e+1}\subseteq\Delta\cap B_i$, So we have $\psi(i-e+1)=d(i-e+1)+s\ge d(i+1)+\ord(\fp)$ and hence $s\ge ed+\ord(\fp)$. And by Proposition \ref{di-prop2}, $a=s-ed-\ord(\fp)\ge0$.
\end{remark}

\section{Applications of $\fp$-difference index}
\subsection{The Hilbert-Levin regularity}
For a $\D$-prime ideal $\fp$, the polynomial $\varphi(i)=\D\textrm{-}\dim(\fp)(i+1)+\ord(\fp)$ is known as the dimension polynomial of $\fp$ (see for instance \cite[Chapter 5]{wibmer}). The minimum of the indices $i_0$ such that $\varphi(i)=\trdeg_K(\Frac(A_i/(A_i\cap\fp)))$ for all $i\ge i_0$ is called the {\em Hilbert-Levin regularity} of $\fp$. The results developed on $\fp$-difference indices enable us to give an upper bound of the Hilbert-Levin regularity of $\fp$.
\begin{theorem}
Suppose $F$ is a difference algebraic system which is quasi-prime at $\fp$. Let $\rho$ and $\omega$ be the $\fp$-quasi regularity degree and the $\fp$-difference index of the system $F$ respectively. $\fp$ is a minimal $\D$-prime ideal over $[F]$. Then the Hilbert-Levin regularity of the $\D$-prime ideal $\fp$ is bounded by $e-1+\max\{0,\rho-\omega\}$.
\end{theorem}
\begin{proof}
Since for all $i\in\N$, we have $\Frac(A_i/(A_i\cap\fp))=\Frac(B_i/(B_i\cap\Delta))$. Therefore, $\trdeg_K(\Frac(A_i/(A_i\cap\fp)))=\trdeg_K (\Frac(B_i/(B_i\cap\Delta)))$ and so, it suffices to show that for all $i\ge e-1+\max\{0,\rho-\omega\}$, $\trdeg_K(\Frac(B_i/(B_i\cap\Delta)))+\D\textrm{-}\dim(\fp)=\trdeg_K(\Frac(B_{i+1}/(B_{i+1}\cap\Delta)))$.

Fix an index $i\ge e-1+\max\{0,\rho-\omega\}$. Since $i+\omega\ge\rho+e-1$, by Theorem \ref{mc-tm}, we have that $\Delta\cap B_i=\Delta_{i-e+1+\omega}\cap B_i$
and $\Delta\cap B_{i+1}=\Delta_{i-e+2+\omega}\cap B_{i+1}$. Thus, by Proposition \ref{di-prop1}, we obtain:
\begin{align*}
\trdeg_K(\Frac(B_{i+1}/(\Delta\cap B_{i+1})))&=d(i+2)+(d+r-n)\omega+s-ed-\mu_{\omega},\\
\trdeg_K(\Frac(B_{i}/(\Delta\cap B_{i})))&=d(i+1)+(d+r-n)\omega+s-ed-\mu_{\omega},
\end{align*}
where $d=\D\textrm{-}\dim(\fp)$ by Proposition \ref{di-prop2}. Hence, the result holds.
\end{proof}

\subsection{The ideal membership problem}
It is well-known that in polynomial algebra, the ideal membership problem is to decide if a given element $f\in A$ belongs to a fixed ideal $I\subseteq A$ for a polynomial ring $A$, and if the answer is yes, representing $f$ as a linear combination with polynomial coefficients of a given set of generators of $I$.

The ideal membership problem also exists in differential algebra and difference algebra. But unlike the case in polynomial algebra, this problem is undecidable for arbitrary ideals in differential algebra (see \cite{ga-crd}) and difference algebra. However, there are special classes of differential ideals for which the problem is decidable, in particular the class of radical differential ideals (\cite{etda}, see also \cite{rdi}).

When it comes to the representation problem, the differential case or the difference case involves another additional ingredient: the order $N$ of derivatives or transforms of the given generators of $I$ needed to write an element $f\in I$ as a polynomial linear combination of the generators and their first $N$ total derivatives or total transforms. The known order bounds seem to be too big, even for radical ideals (see for instance \cite{gp-bod}, where an upper bound in terms of the Ackerman function is given, or \cite{gu}, a better and more explicit upper bound). In \cite{iqr}, an order bound for quasi-regular differential algebraic systems is given, due to the properties of differential indices defined in the same paper. However, it seems that there does not exist any results on the corresponding bound in the difference case except a bound for quasi-regular difference algebraic systems in \cite{jie}. By virtue of Theorem \ref{mc-tm}, we are able to give an order bound for the ideal membership problem of a quasi-prime difference system. 

The following ideal membership theorem for polynomial rings will be used.
\begin{theorem}(\cite[Thoerem 3.4]{as})\label{adi-thm}
Let $K$ be a field and $g_1,\ldots,g_s\in K[y_1,\ldots,y_n]$ be a couple of polynomials whose total degrees are bounded by an integer $d$. Let $g$ be a polynomial belonging to the ideal generated by $g_1,\ldots,g_s$, then there exist polynomials $a_1,\ldots,a_s$ such that $g=\sum_{j=1}^sa_jg_j$ and $\deg(g_j)\le (2d)^{2^n}$ for $1\le j\le s$.
\end{theorem}

We have the following effective ideal membership theorem for quasi-prime difference algebraic systems:
\begin{theorem}
Suppose $F$ is a quasi-prime difference algebraic system in the sense of Remark \ref{pd-re}. Let $\rho$ and $\omega$ be the quasi regularity degree and the difference index of the system $F$ respectively. Let $D$ be an upper bound for the total degrees of $f_1,\ldots,f_r$. Let $f\in K\{\Y\}$ be any $\D$-polynomial in the $\D$-ideal $[F]$ such that $\omega+\max\{0, \ord(f)-e+1\}\ge\rho$. Set $N:=\omega+\max\{-1,\ord(f)-e\}$. Then, a representation
\[f=\sum_{1\le i\le r,0\le j\le N}g_{ij}f^{(j)}_i\]
holds in the ring $A_{N+e}$, where each polynomial $g_{ij}$ has total degree bounded by $(2D)^{2^{(N+e+1)n}}$.
\end{theorem}
\begin{proof}
The upper bound on the order of transforms of the polynomials $f_1,\ldots,f_r$ is a direct consequence of Theorem \ref{mc-tm} applied to $i:=\max\{e-1, \ord(f)\}$. The degree upper bound for the polynomials $g_{ij}$ follows from Theorem \ref{adi-thm}.
\end{proof}
\begin{remark}
Since we have an upper bound $e(\min\{r,n\}+2)$ for $\omega$, it suffices to take $N=e(\min\{r,n\}+2)+\max\{-1,\ord(f)-e\}$ to get more explicit upper bounds of the order and the degree in the above ideal membership problem.
\end{remark}

\section{An example}
\begin{example}
Notations follow as before. Consider the difference algebraic system $F=\{y_1^{(2)}-y_1,y_1^{(1)}-y_2,y_1y_2-1\}\subseteq A=K\{y_1,y_2\}$. Then $\Delta=[F]$ is a $\D$-prime ideal and $F$ is a quasi-prime system in the sense of Remark \ref{pd-re}. We have $n=2,r=3,e=2,d=0$. The corresponding matrices $J_{k},k=1,2,3,\ldots$ are
\[\begin{pmatrix}
-1&0&0&0&1&0&&&&&&\\0&-1&1&0&0&0&&&&&&\\y_2&y_1&0&0&0&0&&&&&&\\&&-1&0&0&0&1&0&&&&\\&&0&-1&1&0&0&0&&&&
\\&&y_1&y_2&0&0&0&0&&&&\\&&&&-1&0&0&0&1&0&&\\&&&&0&-1&1&0&0&0&&
\\&&&&y_2&y_1&0&0&0&0&&\\&&&&&&&&&&&\\&&&&&&\cdots&&\cdots&&\cdots&\\&&&&&&&&&&&
\end{pmatrix},\]
and $J_{k,1},k=1,2,3,\ldots$ are
\[\begin{pmatrix}
1&0&&&&&&&&&&\\0&0&&&&&&&&&&\\0&0&&&&&&&&&&\\0&0&1&0&&&&&&&&\\1&0&0&0&&&&&&&&\\0&0&0&0&&&&&&&&
\\-1&0&0&0&1&0&&&&&&\\0&-1&1&0&0&0&&&&&&\\y_2&y_1&0&0&0&0&&&&&&\\&&-1&0&0&0&1&0&&&&\\&&0&-1&1&0&0&0&&&&
\\&&y_1&y_2&0&0&0&0&&&&\\&&&&-1&0&0&0&1&0&&\\&&&&0&-1&1&0&0&0&&
\\&&&&y_2&y_1&0&0&0&0&&\\&&&&&&&&&&&\\&&&&&&\cdots&&\cdots&&\cdots&\\&&&&&&&&&&&
\end{pmatrix}.\]
Since $y_1^{(2i)}=y_1,y_1^{(2i+1)}=y_2,y_2^{(2i)}=y_2,y_2^{(2i+1)}=y_1$ in the ring $A/\Delta$ for all $i\in\N$, we have replaced $y_1^{(2i)},y_1^{(2i+1)},y_2^{(2i)},y_2^{(2i+1)}$ by $y_1,y_2,y_2,y_1$ respectively in $J_k$ and $J_{k,1}$ for all $i\in\N$. It can be computed that $\rank(J_{1})=3$, $\rank(J_{2})=5$, $\rank(J_{3})=7$. In fact, $\rank(J_{k})=2k+1$ for all $k\ge1$. So the quasi dimension polynomial of the system $F$ is $\psi(k)=2k+1$ and the quasi regularity degree $\rho=1$. Also, one can compute that $\rank(J_{1,1})=1$, $\rank(J_{2,1})=2$, $\rank(J_{3,1})=4$, $\rank(J_{4,1})=6$, so $\mu_1=2,\mu_2=4,\mu_3=5,\mu_4=6$. In fact, $\mu_k=k+2$ for all $k\ge 2$. Hence the difference index of the system $F$ is $\omega=2$. One can check that $\Delta_2\cap A_1=\Delta\cap A_1$.
\end{example}

{\bf Acknowledgements.} The author thanks Li Wei for helpful suggestions.


\begin{thebibliography}{99}
\bibitem{iqr}
L. D'Alfonso, G. Jeronimo, G. Massaccesi, P. Solern\'o.
{\em On the Index and the Order of Quasi-regular Implicit Systems of Differetial Equations},
Linear Algebra Appl., 430(8-9), 2102-2122, 2009.

\bibitem{ldi}
L. D'Alfonso, G. Jeronimo, P. Solern\'o.
{\em A linear algebra approach to the diferential index of generic DAE systems},
Appl. Algebra Engrg. Comm. Comput., 19(6), 441-473, 2008.

\bibitem{as}
M. Aschenbrenner.
{\em Ideal membership in polynomial rings over the integers},
J. Am. Math. Soc., 17(2), 407-441, 2004.

\bibitem{rdi}
F. Boulier, D. Lazard, F. Ollivier, M. Petitot.
{\em Representation for the radical of a finitely generated differential ideal},
Proc. of ISSAC, 158-166, 1995.

%\bibitem{d-br}
%K. Brenan, S. Campbell, L. Petzold.
%Numerical Solution of Initial-Value Problems in Differential-Algebraic Equations,
%{\em SIAM¡¯s Classics in Applied Mathematics}, Philadelphia 1996.

\bibitem{d-ca}
S. Campbell, W. Gear.
{\em The index of general nonlinear DAE¡¯s},
Numer. Math., 72, 173¨C196, 1995.

\bibitem{es-ca}
D. Eisenbud.
{\em Commutative Algebra With a View Toward Algebraic Geometry},
Springer-Verlag, New Work, 2004.

%\bibitem{d-fl}
%M. Fliess, J. L\'evine, P. Martin, P. Rouchon.
%Implicit Differential Equations and Lie-B\"acklund mappings,
%{\em Proc. of the 34th. Conf. on Decision \& Control}, New Orleans, 2704-2709, 1995.

\bibitem{ga-crd}
G. Gallo, B. Mishra, F. Ollivier.
{\em Some Constructions in Rings of Differential Polynomials},
Proc. of AAECC-9, Lecture Notes in Computer Science, Springer-Verlag, 539, 171-182, 1991.

\bibitem{gp-bod}
O. Golubitski, M. Kondratieva, A. Ovchinnikov, A. Szanto.
{\em A Bound for Orders in Differential Nullstellensatz},
J. Algebra, 322(11), 3852-3877, 2009.

\bibitem{gu}
R. Gustavson, M. Kondratieva, A. Ovchinnikov.
{\em New effective differential Nullstellensatz},
Adv. Math., 290, 1138-1159, 2016.

%\bibitem{ko}
%J. Kollar.
%Sharp Effective Nullstellensatz,
%{\em Jou. Ame. Math. Soc.}, 1(4), 963-975, 1988.

\bibitem{d-le}
G. Le Vey.
{\em Differential Algebraic Equations: a new look at the index},
Rapp. Rech., 2239, INRIA, 1994.

\bibitem{d-pa}
C. Pantelides.
{\em The consistent inicialization of differential-algebraic equations},
SIAM Journal of Sci. and Stat. Comput., 9(2), 213-231, 1988.

%\bibitem{d-po}
%M. Poulsen.
%Structural Analysis of DAEs,
%Ph.D. Thesis. Technical University of Denmark, 2001.

%\bibitem{d-pr}
%F.L. Pritchard, W.Y. Sit.
%On Initial Value Problems for Ordinary Differential-Algebraic Equations,
%{\em Radon Series Comp. Appl. Math}, 1, 1-57, 2007.

%\bibitem{d-ra}
%P. Rabier, W. Rheinboldt.
%A Geometric Treatment of Implicit Differential-Algebraic Equations,
%{\em J. of Diff. Equations} 109, 110-146, 1994.

%\bibitem{d-re}
%G.J. Reid, P. Lin, A.D. Wittkopf.
%Differential Elimination-Completion Algorithms for DAE and PDAE,
%{\em Stud. Appl. Math.} 106, 1-45, 2001.

%\bibitem{re}
%B. Renschuch.
%{\em Beitr\"age zur konstruktiven Theorie der Polynomideale. XVII/1. Zur Hentzelt/Noether/Hermannschen Theorie der endlich vielen Schritte}, 
%Wiss. Z. P\"adagog. Hochsch. ``Karl Liebknecht" Potsdam 24, no. 1, 87¨C99, 1980.

\bibitem{etda}
A. Seidenberg.
{\em An elimination theory for differential algebra},
Univ. California Publ. Math., 31-38, 1956.

\bibitem{d-se}
W. Seiler.
{\em Indices and Solvability of General Systems of Differential equations},
Comput. Algebra in Scientific Comput., CASC 99, Springer, 365-385, 1999.

\bibitem{jie}
J. Wang.
{\em Difference Index of Quasi-regular Difference Algebraic Systems}, 
Preprint, arXiv:1607.04076, 2016.

\bibitem{wibmer}
M. Wibmer.
{\em Algebraic Difference Equations}, 
Preprint, 2013.
\end{thebibliography}
\end{document}